\documentclass{amsart}
\usepackage{amsmath, amsthm}
\usepackage{amssymb, latexsym}
\usepackage{amsfonts}

\newtheorem{theorem}[equation]{Theorem}

\newtheorem{proposition}[equation]{Proposition}
\newtheorem{lemma}[equation]{Lemma}

\newtheorem{corollary}[equation]{Corollary}
\newtheorem{conjecture}[equation]{Conjecture}

\newcounter{com}

\newtheorem*{A'}{Corollary A$'$}

\newtheorem*{C}{Theorem C} 

\newtheorem*{C'}{Theorem C'}
\newtheorem*{C''}{Theorem B$''$ (Holt's Conjecture)}

\newtheorem*{CC}{Conjecture C}

\newcommand{\beql}[1]{\begin{equation}\label{#1}}
\newcommand{\eeq} {\end{equation}}

\font\Aaa=msam10

\def\qed{\hbox{~~\Aaa\char'003}}

\font\Bbb=msbm10
\def\semi{\hbox{\Bbb o}}

\def\Z{\hbox{\Bbb Z}}

\def\F{\hbox{\Bbb F}}

\def\CC{\hbox{\Bbb C}}
\def\hom{\hbox{\rm{Hom}}}
\def\hr{\hat{r}}

\numberwithin{equation}{section}

\let\define=\def
\let\redefine=\def

\define\({ \left( }
\define\){ \right) }
\define\[{ \left[ }
\define\]{ \right] }
\define\<{ \langle }
\define\>{ \rangle }
\let\ljunk=\{
\let\rjunk=\}
\redefine\{{\left\ljunk}
\redefine\}{\right\rjunk}

\redefine\.{ \diamomd }
\redefine\-{ \ominus }
\redefine\+{\oplus}
\redefine\.{ \cdot}
\redefine\#{\sharp}

\redefine\.{ \cdot }

        \def\PSL{{\rm PSL}}
        \def\SL{{\rm SL}}

        \def\AGL{{\rm AGL}}

        \font\Aaa=msam10

\DeclareRobustCommand{\SkipTocEntry}[4]{}
\begin{document}

\title[Proficient Groups]
{Remarks on Proficient groups }

\thanks{The authors were partially supported by
        NSF grants DMS 0653873,
          DMS 0242983, DMS 0600244  and   DMS~0354731.}

       \author{R. M. Guralnick}
       \address{Department of Mathematics, University of Southern California,
       Los Angeles, CA 90089-2532 USA}
       \email{guralnic@usc.edu}

       \author{W. M. Kantor}
       \address{Department of Mathematics, University of Oregon,
       Eugene, OR 97403 USA}
       \email{kantor@math.uoregon.edu}

    \author{M. Kassabov}
       \address{Department of Mathematics, Cornell University,
Ithaca, NY 14853-4201  USA}
       \email{kassabov@math.cornell.edu}

    \author{A. Lubotzky}
       \address{Department of Mathematics, Hebrew University, Givat Ram,
Jerusalem 91904 Israel}
       \email{alexlub@math.huji.ac.il}


{\abstract   
If a finite group  $G$ has a presentation with $d$ generators and
$r$ relations, it is well-known that $r - d $ is  at  least the rank of
the Schur multiplier of $G$; a presentation is called {\em efficient}
if equality holds.    There is an analogous definition for
{\em proficient}  profinite
presentations.   
  We show that  many perfect groups
have proficient presentations.  Moreover, we prove
that infinitely many alternating groups, symmetric groups and their double
covers  have proficient presentations. }

\maketitle



\centerline{Dedicated to the memory of Karl Gruenberg}

\section{Introduction}
\label{intro}

For a group $H$, we denote by
$d(H)$ the minimal number of generators of $H$.
If $N\lhd H$, we denote by $d_H(N)$ the
minimal number of generators of $N$ as
a normal subgroup of $H$.

A finite group $G$ has a \emph{presentation with $d$
generators and $r$ relations} if there is an
exact sequence
\begin{equation} \label{basic discrete}
1\rightarrow R  \rightarrow F \rightarrow G\rightarrow 1,
\end{equation}
where  $F$  is a free group on $d$ generators and
$d_F(R)=r$.
Similarly, $G$ has a \emph{profinite   presentation with
$d$ generators and $r$ relations} if there is  an exact sequence
\begin{equation} \label{basic profinite}
1\rightarrow \widehat{R}  \rightarrow \widehat{F} \rightarrow G\rightarrow 1,
\end{equation}
where $\widehat{F}$ is the free profinite group on $d$
generators and $d_{\widehat F}(\widehat{R})=r$;
here $d_{\widehat F}(\widehat{R})$ is the minimal number
of normal generators of $\widehat{R}$ in 
the
topological sense, i.e.,
as a normal closed subgroup of $\widehat{F}$.

It is known  (\ref{prof ineq}) that, if $G$ has a (profinite)
presentation with $d$ generators and $r$ relations, then
$ r -d  \geq d(M(G))$,
where $M(G)$ is the Schur multiplier of $G$.
A presentation (resp.~profinite presentation) is
called {\em efficient} (resp. {\em proficient})  if  $ r -d  = d(M(G))$;  and
$G$ is called {\em efficient} (resp. {\em proficient}) if 
it has an efficient (resp. proficient) presentation.
It is also known  
 that if a finite group $G$ has a
proficient presentation, then 
it has a proficient presentation with only $d(G)$ generators
(cf. Proposition} ~\ref{lub0.1}).
The analogous result in the category of discrete groups
is an old open problem (cf.  \cite[p.~2]{relmod}).

The notion of efficient presentations is relatively old,
but the list of  perfect groups  or $2$-generated groups 
known to have such presentations is  very  limited.
   The only infinite family of simple groups presently known  to 
 have efficient presentations consists of the groups  $\PSL(2,p)$ with $p > 3$ prime \cite{Sun};
 $\SL(2,p),$  $ p > 3$, also has an efficient presentation \cite{CR2}.
 In addition,  $\PSL(2,p) \times \PSL(2,p)$ has an efficient presentation for each prime $p>3$
 \cite{pp},  as do 
 $\SL(2, p) \times \PSL(2, p), $   $ \PSL(2, p)^3$,   
$ \PSL(2, p)\times A_6 $,   $\PSL(2, 5)^4$
 and  most  ``small''  simple groups 
 \cite{Ro,CMRW,CR3, CRKMW,CHRR,CHRR2}.  Also $\SL(2,\mathbb{Z}/m)$ is efficient
for any odd   integer $m$ \cite[p.~19]{CR2} (compare \cite[p.~70]{CRW88}), and hence
so  is any quotient by a subgroup of
its center.    On the other hand, Harlander \cite[Cor. 5. 4]{Ha} 
has shown that, for any finite group
$G$,  $G \times P$ is efficient for a sufficiently large elementary abelian
$p$-group $P$ (in particular, every perfect group is the derived group of an efficient group).
Note that these groups have a very large number of generators and a much larger
number of relations.  See also \cite{El}.

The notion of proficient presentations was introduced by
Gruenberg and Kov\'acs in \cite{GrK}.
An efficient presentation gives rise to a proficient one,
so all efficient groups are proficient. The
present note is an offshoot of our result in \cite{GKKL3} 
that, for primes $ p \equiv 11 \mod 12,  A_{p+2 } \times T$ has an efficient presentation,
where $T$ is the subgroup of index $2$  in $AGL(2,p)$. 
Combined with the cohomological methods of \cite{GKKL2} we will provide further examples
of proficient groups which are perfect (or  very close  to   perfect).   Indeed,   for any $d > 1$ we 
provide infinitely many examples of perfect groups  $G$   such
that $G$ has a profinite presentation with $d=d(G)$ generators and $d$ relations
(see Corollary \ref{0 prodeficiency}). 
By contrast,   there appear to be no examples known of finite groups that
have presentations with $d(G)$ generators and $d(G)$ relations     when $d(G)> 3$.
By the Golod-Shavarevich Theorem \cite{serre}, this cannot occur for finite nilpotent groups.

For a finite group $G$ we denote by $r(G)$ (resp. $\hat r(G)$)
the minimal number of relations needed to define $G$, i.e.,
the minimum of 
$d_F(R)$ over all presentations (\ref{basic discrete}) of $G$
(or of  $d_{\hat F}(R)$ in (\ref{basic profinite})). 
Clearly,
\begin{equation} \label{inequality}
\hat r(G) \leq r(G).
\end{equation}
It is a central open problem in the
area of presentations of finite groups
whether (\ref{inequality})  is always an equality (cf.  \cite[p.~2]{relmod}).
Indeed, Serre \cite[p.~34]{serre} stated  that for 30 years he had seen 
 ``no reason  $\dots$
why this should always  be an equality''.
As a special case, in view of the results in this paper  it is 
especially interesting  to 
 ask whether there are proficient finite groups
that are not efficient. 

We recall  a cohomological interpretation of $\hat r(G)$  (see \cite{GrK, lub34}).
If $M$ is a finite-dimensional $kG$-module for a field $k$, define
$$
\nu_2(M):=\left\lceil\frac{\dim H^2(G,M) - \dim H^1(G,M) + \dim H^0(G,M)}{\dim M}\right\rceil.
$$
Then 
\begin{equation} \label{rhat} 
\hr(G) -  d(G)  =   \max_{p,M}  \nu_2(M) -1 ,
\end{equation} 
where  $p$ runs
over all primes   and $M$ runs over all irreducible $\F_pG$-modules.    
It is well-known (see Lemma \ref{schur}) that the rank of the 
Sylow $p$-subgroup  of $M(G)$ is 
$\dim H^2(G,\F_p) - \dim H^1(G,\F_p)= \nu_2(\F_p) - 1$,
where $\F_p$ is viewed as the trivial $\F_pG$-module. 
Hence,
\begin{equation} \label{nuschur}
d(M(G))= \max_p \nu_2(\F_p) - 1.
\end{equation} 
Thus (\ref{rhat}) and (\ref{nuschur}) imply that
\begin{equation} \label{prof ineq}
\hat{r}(G) \geq d(G) + d(M(G)) .
\end{equation}
By definition,  a group $G$ is proficient if and only if this inequality 
 is  an equality. By (\ref{rhat}),  this happens
if and only if
\begin{equation} \label{prof ineq2}
 \max_{p,M} \nu_2(M) \le  \max_p \nu_2(\F_p) = 1 +
 \max_p \Big(\dim H^2(G,\F_p)-\dim H^1(G, \F_p)  \Big), 
\end{equation} 
 where $p$ runs over all primes and $M$  over all nontrivial irreducible
$\F_pG$-modules.
That is, $G$ \emph{is proficient if and only if 
$ \max_{p,M} \nu_2(M)$
is attained when $M$ is the trivial module $\F_p$ for some prime $p$.}
Thus, for proficiency but not efficiency, we have a cohomological
interpretation    that is crucial in this paper.

 While many finite groups are proficient  
(cf. Proposition \ref{nilpotent} for all finite  nilpotent groups),  
not all finite solvable groups are
(see \cite{swan}  and  \S\ref{non}).
 
In \cite[(2.6)]{GrK}  it is shown that,  if $G$ is any finite group and $H$ is not
a superperfect group (i.e., for some $p$ either
$H^1(H,\F_p)$ or $H^2(H,\F_p)$ is nonzero), then $G \times H^{\times k}$
is proficient for all sufficiently large $k$. In particular, every finite group 
$G$ is a direct factor of a
proficient group (by \cite{Ha} in fact every finite group is a direct summand
of a finite efficient group).   In \S7,  we will see that every finite
group is also a direct factor of  a non-proficient finite group (and so also
a non-efficient one). 

A  method similar to that in \cite{GrK}, 
 combined with our quantitative results from  \cite{GKKL2},
yields:
\begin{theorem} \label{theorem 18}
\begin{enumerate}
\item[\rm(a)]  If $G$ is a direct product of $2$ or more simple 
alternating groups$, $  then $G$ is proficient.
\item[\rm(b)]  If $G$ is a direct product of finite quasisimple groups and if 
$d(M(G)) \ge 16,$  then $G$ is proficient.
\item[\rm(c)]  If $G$ is a direct product of quasisimple groups 
each of which is a covering group of  $\PSL(2,q_i)$ for  some prime power $q_i>3,$ then 
$G$ is proficient.
\end{enumerate}
\end{theorem}

A basic related question  is the following:

\begin{conjecture} \label{conjA}
Every finite simple group $S$  and 
its universal cover $\tilde{S}$ are proficient.
\end{conjecture}

Here it  is clear that,  if $\tilde{S}$ is efficient or proficient,
then so is $S$ (a presentation for $S =\tilde{S}/Z$ is obtained
by taking one for $\tilde{S}$ and killing generators for
$Z \cong M(S)$ -- indeed this observation is  also obvious for finite
perfect groups).

We have already noted that $\tilde{S}$ is proficient if and only if it has a
profinite presentation with $2$ generators and $2$ relations
(cf.~Proposition~\ref{lub0.1}).
Wilson \cite{Wil}  conjectured that $\tilde{S}$ even has such  a
discrete presentation, so his conjecture implies the previous one. 

As
$H^1(\tilde S, \F_p) = H^2(\tilde S, \F_p)=0$
for all $p$, (\ref{prof ineq2})   and  (\ref{rhat}) imply that  Conjecture \ref{conjA}  
is equivalent to:

\begin{conjecture}\label{conjB}
For every finite simple group $S,$ every prime $p$ and every
nontrivial irreducible $\F_p{\tilde S}$-module $M,$
$$
\dim H^2(\tilde S,M) - \dim H^1(\tilde S,M) \leq \dim M.
$$
\end{conjecture}

It is known that $\dim H^1(\tilde S,M)$ is relatively small with respect to
$\dim M$ (in \cite{GH}  it is shown that 
$\dim H^1(\tilde S,M) \le (1/2) \dim M$
for each  finite-dimensional  $\tilde S$-module $M$).
Therefore the following   stronger version of the preceding conjecture   seems   likely:
 
\begin{conjecture}\label{conjB+}
For every finite quasisimple group $S,$  every prime $p$ and every
 $\F_pS$-module $M,$
$$
\dim H^2(\tilde S,M) \leq \dim M.
$$
\end{conjecture}

 Note that if $S$ is a finite quasisimple group and $\tilde  S $
satisfies Conjecture \ref{conjB+}, then it satisfies
the two earlier conjectures as well, and so, as noted above, every homomorphic
image of $\tilde S$ is proficient. 

 Our general result   \cite[Theorem B]{GKKL2}  approximates this:

\begin{theorem}
For every finite quasisimple group $S,$  every prime $p$ and every
 $\F_pS$-module $M,$
$$
\dim H^2(\tilde S,M) \leq (17.5)\dim M.
$$
\end{theorem}

Thus, the preceding  conjecture  drops the constant $17.5$ to $1$, and in some cases 
there are even  
much stronger bounds.
However,  decreasing 17.5   to
1 in general would need new ideas.

In addition to the simple groups known to be efficient (and hence also proficient),
${\rm Sz}(2^{2k+1})$  is proficient and satisfies Conjecture \ref{conjB+}  \cite{Wil}.
By \cite[Theorems 7.2, 7.3]{GKKL2},    $\SL(2,q)$ and 
$\PSL(2,q)$ also satisfiy  Conjecture \ref{conjB+}, whence they are proficient.

The main results of the present paper add  more groups to this list.  
For example, 
by Theorem~\ref{H2 estimate Alt Sym},
if $p \equiv 3 \mod 4$ is prime   then $A_{p+2}$ satisfies
Conjecture \ref{conjB+}:

\begin{theorem} \label{altcoho} If $p \equiv 3 \mod 4$ is prime
and $X$ is a quasisimple group with $X/Z(X) \cong A_{p+2} ,$
then $\dim H^2(X,M) \le \dim M$.
\end{theorem}

In particular, this gives the first known examples of 
proficient simple groups
(and universal covers) where the ``rank'' 
goes  to infinity -- viewing alternating and symmetric groups as groups of Lie type over ``the field of order $1$"
\cite{Tits}.

\begin{theorem}  \label{altp+2}
Let $p$ be a prime.~\begin{enumerate}
\item[{\rm(a)}]
 If $p \equiv 2 \mod 3$ then $S_{p+2}$ is proficient$:$   it has a
profinite presentation with $2$ generators and $3$ relations$;$
\item[{\rm(b)}]
 If $p \equiv 3 \mod 4$ then $A_{p+2}$ and its double cover
$2A_{p+2}$ are proficient$:$
$A_{p+2}$  $($resp. $2A_{p+2})$ has a
profinite presentation with $2$ generators and $3$ $($resp. $2)$ relations$;$
\item[{\rm(c)}]  If $p \equiv 3 \mod 4$ then $S_{p+2}$ and either double cover
$2S_{p+2}$ are proficient$:$
$S_{p+2}$ $($resp. $2S_{p+2})$  has a
profinite presentation with $2$ generators and $3$ $($resp. $2)$ relations$;$
\item[{\rm(d)}]
 $\SL(2,q)$ is proficient for  every prime power $q \ge 4; $
\item[{\rm(e)}]
$\PSL(3,q)$ is proficient for every prime power $q \equiv 1 \mod 3;$  and
\item[{\rm (f)}]   $\PSL(4,q)$ is proficient for every odd $q$.
\end{enumerate}
\end{theorem}

While parts (d)-(f) are   immediate applications of our  
results in \cite[Section 7]{GKKL2}, 
parts (a)-(c) require a combination of results on discrete 
presentations from \cite{GKKL3}  together 
with   cohomological arguments.  In particular, we use the 
following result:

\begin{proposition} \label{SxT}
Let $G = X \times Y$ be a finite group. Suppose $G$ has a
profinite presentation with $d$ generators and $r$ relations.
If $d(X)=d'$ then $X$ has a profinite presentation with $d'$ generators
and $ r - (d-d')$ relations.
\end{proposition}

We do not know if the analogue of this proposition holds for discrete presentations.
This  is an interesting subcase of the question whether 
or not $r(G) = \hat r(G)$.  

A trivial consequence of the K\"unneth formula and Theorem \ref{altp+2}(b)(d) gives the following:

\begin{corollary} \label{0 prodeficiency}   
For $1 \le i \le t$ let $G_i $ be either $ A_{p_i+2}$ with $p_i \equiv 3 \mod 4$ prime or $\SL(2,q_i), q_i > 4$.   
Let $G=G_1 \times \ldots \times G_t$.   Then  $G$ has a profinite presentation with $d(G)$ generators and $d(G)$ relations.
In particular, for any integer $d > 1$, there exist infinitely many  finite perfect groups $G$ that have a profinite presentation
with $d=d(G)$ generators and $d(G)$ relations.
\end{corollary}
 
In fact, fix  $S$ to be one of the quasisimple groups in the corollary.  Then for any integer $d > 1$, there is some $t$
such that    $d(S^t)=d$ and so $S^t$ has a profinite presentation with $d$ generators and $d$ relations.
 
This paper is organized as follows.  In \S \ref{coho1}, 
we give some preliminary results
on cohomology.  In \S \ref{dp}, we  deduce  Proposition~\ref{SxT}
and prove   other results on direct products.   Combining this
with results in  \cite{GKKL2} proves Theorem \ref{theorem 18}.
  In  \S \ref{dp}   we also prove parts of Theorem \ref{altp+2}.
In \S \ref{discrete}, we give   discrete presentations for groups related
to covers of alternating groups.  In \S \ref{coho2}, 
we prove further  results about
cohomology (in particular about the cohomology of amalgamated products).
In \S \ref{altsec}, we use our results on discrete presentations and 
cohomology to prove Theorem \ref{altcoho} in characteristic 2;   this bound was already
proved in odd  characteristic in \cite[Theorem 6.2]{GKKL2}. We then complete the proof 
of Theorem \ref{altp+2}. In the final section,
we give a general construction  of non-proficient perfect groups.

This paper is dedicated to the memory of  Karl Gruenberg whose major contributions
to the subject of discrete and profinite presentations have been an inspiration 
to many. 
 
 \section{Cohomology and preliminaries}  \label{coho1}
 
 If $G$ is a finite nilpotent group and $M$ is an irreducible $\F_pG$-module,
 then either $M$ is trivial or some normal $p'$-subgroup of $G$ acts without
 fixed points.  In the latter case, $H^i(G,M)=0$  for all $i$ (see \cite[Cor. 3.12(2)]{GKKL2}).
 Thus, we have the following trivial consequence  of (\ref{prof ineq2}):
 
 \begin{proposition} \label{nilpotent}  All finite nilpotent groups are proficient.
 \end{proposition}

The \emph{inflation restriction sequence}  will be used frequently 
(see \cite[2.6]{relmod}):

\begin{lemma}  \label{infres} Let 
$M$ a $\Z G$-module
for   a $($possibly infinite$)$ group $G$. 
If $N$ is a normal subgroup of  $G,$ then  there is an exact sequence
$$
0 \rightarrow H^1(G/N,M^N) \rightarrow H^1(G,M) \rightarrow H^1(N,M)^G
\rightarrow H^2(G/N,M^N) \rightarrow H^2(G,M).
$$
 \end{lemma}
 
 We will most often use this when $N$ acts trivially on $M$,
 in which case $M^N=M$ and $H^1(N,M)^G = \hom_{G/N}(N/[N,N],M)$.

We next  recall a well-known result about
the rank of the Schur multiplier  $M(G)=H^2(G, \CC^*)$
of a finite group  $G$:
\begin{lemma} \label{schur}  Let $G$ be a finite group and $p$ a prime.
Then $\dim H^2(G,\F_p) - \dim H^1(G,\F_p)$ is equal to the rank of the
Sylow $p$-subgroup of $M(G)$.   In particular$,$
$\dim H^2(G,\F_p) \ge \dim H^1(G,\F_p)$.
\end{lemma}

\begin{proof} Consider the short exact sequence 
$0 \rightarrow \F_p \rightarrow \CC^* \rightarrow \CC^* \rightarrow 0$ of $\Z G$-modules, 
where the map on $\CC^*$ is the $p$th power map.  The long exact sequence
for cohomology \cite[III.6.2]{Br} gives
$
0 \rightarrow H^1(G,\F_p) \rightarrow H^1(G,\CC^*) \rightarrow H^1(G,\CC^*)
\rightarrow H^2(G,\F_p) \rightarrow H^2(G,\CC^*) \rightarrow pH^2(G,\CC^*) \rightarrow 0$.
Thus, $|H^1(G,\F_p)||H^2(G,\CC^*)\colon \! pH^2(G,\CC^*)|=|H^2(G,\F_p )|$, 
which completes the proof. 
\end{proof}

 Serre \cite[Prop. 28]{serre} gave a proof of the final statement
for finite $p$-groups (although his proof is valid for all finite groups).
Of course, the Golod-Shavarevich Theorem implies that  
$\dim H^2(G,\F_p)- \dim H^1(G,\F_p)$ is usually very large for $p$-groups.

\begin{lemma}  \label{h2-h1 nontrivial}  Let $G$ be a finite group and 
$N\unlhd G$.   Let $V$ be an
irreducible $\F_pG$-module with $N$ trivial on $V$. 
Assume that there is no nontrivial 
$G$-equivariant homomorphism from $N$ onto $V$.  Then
\begin{enumerate}
\item  $\dim H^1(G/N,V) = \dim H^1(G,V),$
\item  $\dim H^2(G,V) \ge \dim H^2(G/N,V),$  and
\item    $\dim H^2(G,V) - \dim H^1(G,V) \ge \dim H^2(G/N,V) - \dim H^1(G/N,V)$.
\end{enumerate}
\end{lemma}

\begin{proof}  (1) and (2) follow from
  the inflation restriction sequence (Lemma \ref{infres}):
$$
0 \rightarrow H^1(G/N,V) \rightarrow H^1(G,V) \rightarrow \hom_G(N,V)
\rightarrow H^2(G/N,V) \rightarrow H^2(G,V). 
$$
Then (3) follows from (1) and (2). 
\end{proof} 
  
  Finally, we state a very useful  consequence of \cite[Theorem~0.1]{lub34}:
\begin{proposition}  \label{lub0.1}
If $G$ is a finite group having a  profinite  presentation with $d$ generators 
and $r$ relations$,$ 
 then $G$ also has a profinite presentation with $d_0:=d(G)$ generators and $r_0$ relations 
for which 
$r_0-d_0\le r-d$.
\end{proposition} 

This is one of the tools that make profinite presentations easier to work 
with than discrete presentations --  and we do not know whether or not the discrete 
analogue holds.

\section{Direct  products} \label{dp} 

We can now prove Proposition \ref{SxT}:

\begin{proof}  We are assuming that $G$ has a profinite presentation
with $d$ generators and $r$ relations.  
By Proposition~\ref{lub0.1} we may assume that
$d$ and $r$ are both minimal.  
Assume that $X$ has a profinite presentation with  $d'=d(X) \le d$ generators   
and $r'$ relations with $r'$ minimal (again we use Proposition~\ref{lub0.1}).

By  (\ref{rhat}),  $r' - d'$ is   the maximum of
$\nu_2(M) - 1$
as $p$ ranges over all primes and  $M$ ranges over all irreducible 
$\F_pX$-modules.   Similarly,   $r-d$ is defined by the
same formula in terms of irreducible $G$-modules.

We need to show that $r'-d' \le r - d$,   and so
 it suffices to prove that 
 $$\dim H^2(G,M) - \dim H^1(G,M) \ge \dim H^2(X,M)-\dim H^1(X,M)$$
for   every prime $p$ and  every irreducible $\F_pX$-module    (where
we view $M$ as an  $\F_pG$-module with $Y$ acting trivially). 

First suppose
that $M$ is a nontrivial $\F_pX$-module. Then $M$ is not a homomorphic image of $Y$
(since $X$ acts trivially on $Y$ but not on $M$), whence Lemma \ref{h2-h1 nontrivial} implies
the desired inequality  (using $N=Y $ and $V=M$).   

If $M = \F_p$, the desired inequality follows
from Lemma \ref{schur}  
since the Schur multiplier of $X \times Y$ contains the
Schur multiplier of $X$.
\end{proof}  

This produces  an extension of  \cite[2.8]{GrK}:

\begin{corollary}  If $X \times Y$ is a proficient finite group
and $d(M(X))=d(M(X \times Y)),$  then $X$ is proficient.
\end{corollary} 
\begin{proof} 
By Proposition~\ref{SxT} and hypothesis, $d(M(X \times Y)) =
r(X\times Y)-d(X \times Y)\ge r(X)- d(X)\ge d(M(X)) =d(M(X \times Y))  $.
\end{proof}  
 
We do not know whether or not the previous two results hold for
discrete presentations. 
One can extend the corollary by using the same argument
as in the proof of  Proposition~\ref{SxT}  to show
the following:

\begin{corollary}  Let $G$ be a finite group with a normal subgroup
$Y,$  and let $X:=G/Y$.  Assume that there is no $G$-equivariant homomorphism
from $Y$ to a nontrivial irreducible $X$-module.  
Then $\hat{r}(X) - d(X) \le \max \{\hat{r}(G) -d(G) , d(M(X)\}$.

In particular, if $G$ is proficient and $d(M(X))=d(M(G))$,
then $X$ is proficient.
\end{corollary} 


\noindent
\emph{Proof of} Theorem~\ref{altp+2}(a).     By \cite[Corollary 3.13(ii)]{GKKL3}, 
for any   $p \equiv 2 \mod 3$  there is a group $G$
of index $2$ in $S_{p+2} \times \AGL(1,p)$ surjecting onto both factors
that has a presentation with $2$ generators and $3$ relations.  
Let $Y$ be the kernel of the projection of $G$ onto $S_{p+2}$.
Now apply the preceding corollary. 
\qed

\medskip

Similarly, we can also prove parts of Theorem~\ref{altp+2}(b).
Let $T$ be the subgroup of  index $2$ in $\AGL(1,p)$
with $p \equiv 11 \mod12$. 
By \cite[Corollary 3.8(i)]{GKKL3}, $A_{p+2}
\times  T $ has a presentation with $2$ generators and $3$
relations.  By  Proposition \ref{SxT}, it follows that
$A_{p+2}$  a profinite presentation with $2$ generators and $3$
relations.  The remainder of Theorem~\ref{altp+2}(b) will be proved
in \S4, and in \S6  we will prove the remaining parts of 
Theorem~\ref{altp+2}
(e.g.,   primes  $p \equiv 3 \mod4$ are dealt with  in  Theorem~\ref{H2 estimate Alt Sym}).  
 
 The next result is a special case of a result  about direct products
in \cite[2.7]{GrK}.

\begin{lemma} \label{3.3}\label{h11}
 Let $G_i, 1 \le i \le t$   $($for $t\ge 2),$  be finite perfect
groups each of which has a presentation
with $2$ generators and $r_i$ profinite relations. 
Set $X = \prod_{i=1}^t G_i$.   Then
\begin{equation} \label{B} 
  \hat r (X)-d(X)\le \max\{d(M(X)),r_i-2\mid i=1,\dots,t\}.
\end{equation}
    In particular$,$  if
$d(M(X)) \ge \max_i r_i-2 $  then $X$ is proficient.
\end{lemma}

\begin{proof} 
Recall by   (\ref{rhat})  that
\begin{equation} \label{A} 
r_i-2\ge \max_{p,N}\nu_2(N)-1,
\end{equation}
where $p$ runs over all primes and $N$ over all irreducible $\F_pG$-modules.
We know from (\ref{rhat})   that
$\hat r(X)-d(X)=\max_{p,N}\nu_2(M)  -1 $,
where $M$ runs over all irreducible $\F_pX$-modules.
Thus, we have to prove that, for every 
such $M$,
\begin{equation} \label{B0} 
\nu_2(M) \le  \max\{d(M(X))+1, r_i-1 \mid i=1,\dots,t\}. 
\end{equation}

If $M$ is the trivial module then, by (\ref{nuschur}), $\nu_2(M)\le d(M(X))+1$.  
So assume that $M$ is nontrivial.   By \cite[Lemma 3.2]{GKKL2}, 
we may consider modules over a splitting field $F$ for
$X$.  Then $M=\bigotimes_{i=1}^tM_i$ for irreducible 
$FG_i$-modules $M_i$.  
If at least  3 of the $M_i$ are nontrivial then, by the K\"unneth formula 
(cf. \cite[Lemma~3.1]{GKKL2}), 
$H^2(X,M)=0=H^1(X,M)$ and (\ref{B}) holds.

If exactly 2  of the $M_i$, say $M_1$ and $M_2$,  are nontrivial then 
$H^1(X,M)=0$ and   $H^2(X,M)\cong H^1(G_1,M_1) \otimes  H^1(G_2,M_2)$, again by 
the K\"unneth formula.  As the $G_i$ are 2-generated, $\dim H^1(G_i,M_i)\le 
\dim M_i$  and hence  $\dim H^2(X,M)\le 
\dim M$ and $\nu_2(M)\le 1$, so that (\ref{B}) again  holds.

Finally, if only $M_1$ is nontrivial, then the K\"unneth formula gives
$H^2(X,M)\cong H^2(G_1,M_1), H^1(X,M)\cong H^1(G_1,M_1)$ since 
$\dim H^0(G_i,M_i)=1$ and $H^1(G_i,M_i)=0$ for $i>1$.  Thus, this time
$\nu_2(M)=\nu_2(M_1). $  By  (\ref{A}),   $ \nu_2(M_1)\le r_1 -1$, and so again 
  (\ref{B})  holds.
\end{proof}

{\noindent \em Proof of} Theorem \ref{theorem 18}.
By \cite[Theorem B]{GKKL2}, every finite quasisimple group 
$G$ has a profinite  presentation
with $2$ generators and $18$ relations.    
In particular, $r - d(G) \le 16$ for any
finite quasisimple group.
Similarly, by \cite[Theorem D]{GKKL2},
every alternating group has a profinite presentation with 
$2$ generators and $4$
relations.  Also, $\SL(2,q)$ has a profinite
presentation with $2$ generators and $2$ relations
by \cite[\S7]{GKKL2}.  Thus, Theorem \ref{theorem 18} follows
from Lemma~\ref{3.3}.  \qed

\section{Some discrete presentations}  
\label{discrete}

  Carmichael \cite{car} 
proved  that $A_{n+2}$ has a presentation 
\begin{equation}
\label{carmichael}
\langle x_1, \ldots, x_n \mid  x_i^3=1, (x_ix_j)^2 = 1, 1 \le i \ne j \le n \rangle.
\end{equation}
We first observe that this can be modified to give a presentation for the
double cover $2A_{n+2} $:

\begin{proposition} \label{double cover}  If $n \ge 3$ and
$J = \langle x_1, \ldots, x_n \mid x_i^3=(x_ix_j)^2, 
1 \le i \ne j \le n \rangle , $
then $J\cong 2A_{n+2}$.
\end{proposition}
\begin{proof}   There is a surjection  $\phi\colon J \to 2A_{n+2}$ sending
$x_i$ to the element $(i ,  n+1 ,  n+2)z$ of order 6  for $1 \le i \le n$, where $z$ is the 
central involution in $2A_{n+2}$.  Namely, $x_ix_j$ is a product of two disjoint 
transpositions as an element of $A_{n+2}$, and hence has order 4 in $2A_{n+2}$.

Set $w:=x_1^3 =(x_1x_2)^2$  and  $Q := \langle x_1, x_2 \rangle \le J$.
Then $\phi(Q)=2A_4=\SL(2,3)$.
Since $w$ commutes with $x_1$ and $x_1x_2$ it 
is central in $Q$.  Also, in $Q/[Q,Q]$  we  have  $x_1 \equiv x_2^2$
and $x_2 \equiv x_1^2$, and hence  $x_1\equiv x_1^4$, so that $w = x_1^3 \in [Q,Q]$.
Now $Q/\langle w \rangle$ is generated by $2$ elements of order
$3$ whose product is an involution.  
By  (\ref{carmichael}),  $Q/\langle w \rangle \cong A_4$.  
Thus, $Q$ is a cover of $A_4$ and so $Q \cong \SL(2,3)$.
In particular, $x_1$ has order $6$ and $x_2^3=x_1^3$.

Consequently,   $w=x_i^3$ for all $i$ and so   $w$  is a central
involution  in $J$.  Also, $w$    is contained in $[Q,Q] \le [J,J]$.
By  (\ref{carmichael}), $J/\langle w \rangle$
is a homomorphic image of $A_{n+2}$, and so is isomorphic to $A_{n+2}$. 
Thus   $J\cong 2A_{n+2}$.
\end{proof}

We use this to give a   presentation for a group having as a direct factor 
the double
cover of  a suitable alternating group  (cf. \cite[Corollary~3.8(i)]{GKKL3}):
 
\begin{proposition} \label{doublexT} 
For a prime   $p \equiv 11 \mod 12 , $  let    
$$J:=\langle g, u \mid 
 u^p=v^{ (p-1)/2 }, (u ^s)^v=u ^{s-1}
 , (ww^u )^2 =w^3 \rangle,$$
where  $s(e-1)\equiv -1  \mod p$ for an   integer  $e$  having multiplicative order 
$(p-1)/2  \mod  p$,  while $v: = g^{6}$  and $w: = g^{(p-1)/2}$.   
Then $J \cong 2A_{p+2} \times T , $
where $T$ is the subgroup of index $2$ in $\AGL(1,p)$.
\end{proposition}

\begin{proof}
Throughout this section we view $A_p$ as acting on $\F_p$ and 
$A_{p+2}$ as acting on   $\F_p\cup \{ p+1  ,  p+2\}$.
By \cite{bn},     
\begin{equation}
\label{Neumann}
T = \< u_0 ,v_0  \mid u_0 ^p=v_0 ^{ (p-1)/2 }, (u_0 ^s)^{v_0}=u_0 ^{s-1} \>,
\end{equation}
where $u_0 $ corresponds to $x\mapsto x+1$ and $v_0 $ corresponds to $x \mapsto  ex$, 
acting on $\F_p$.

We first show that there is a surjection 
$\phi\colon J \to 2A_{p+2}  \times T$.
Let  $z$  be
the central involution of $2A_{p+2} $.
 Write $T = \langle u_0 , v_0  \rangle$ as above; since it has odd order, its preimage in $2A_p$ 
 has a subgroup we can identify with $T$.
Consequently, we can view $T < 2A_{p+2} $ with   $T$ fixing 
$p+1$ and $ p+2$, while $v_0 $ fixes $0$ as well.
Now define $\phi$ by  $\phi(u) =  (u_0 ,u_0 )$ and $\phi(g)=(v_0 g_0 , v_0 )$, where $g_0: =(0 ,p+1 , p+2)z
\in 2A_{p+2}$
has order 6 and commutes with $v_0$.
Since  $g^6 \mapsto(v_0 ^6,v_0 ^6)$ and  $|v_0^6|=|v_0|$ (recall that $p \equiv 11 \mod 12$),  
this embeds $T$ diagonally
into $\phi(J)$; and since  
$g^{p-1} \mapsto(g_0 ^{ - 1},1)$  $\phi$ is  a surjection.  

Now consider the group $J$.
By (\ref{Neumann})  we can identify $T$ with the subgroup
$\langle u ,v \rangle$.  Since   $v$ centralizes $w$, $\Omega:=w^T$ has size at most $p$, 
and hence the size is $p$ since  $|\phi(\Omega) |\ge p$.
Thus,
$T$ acts on $\Omega$ as it does on  $\F_p$.  In particular, since $p \equiv 3 \mod 4$,  $T$
acts transitively on the 2-element subsets of $\Omega $.

There is   an integer  $k $    such that   $-k$ and $ k - 1 $ are nonzero  
 squares mod~$p$.
We claim that    $x:=w, y:=w^{u} $  and  $  z:=w^{u^{k}}$  satisfy the relations
\begin{equation}\label{xyz}
x^6\! =y^6\! =z^6\! =1, x^3\! =(xy)^2\! =(yx)^2,  y^3\! =(yz)^2\! =(zy)^2, z^3\! =
(zx)^2\! =(xz)^2.%
\end{equation}
The first 3  of these  follow from  $w^6=1$,  
which holds since 
  $ v^{(p-1)/2}=1$  by (\ref{Neumann}).
Moreover, in view of the last relation defining   $J$,  $w$ centralizes $w^3=(w w^u)^2$   
and conjugates  
  $(w w^u)^2$  to $( w^u w)^2$, so $(xy)^2 =( yx)^2$.  
 Note that $T$ has an element sending the ordered pair $(w,w^u)$ to 
$(w^{u^i}, w^{u^j}) $ if and only if $j-i$ is a (nonzero) square mod~$p$.  
In view of our choice of $k$, 
$(w^{u^k}, w) $ and $(w^{u}, w^{u^k}) $ are both in the $T$-orbit of $(w,w^u)$.  
This proves  the last 2 relations in 
 (\ref{xyz}). 
 
The group $\<x,y,z\>$ given by  the   relations  (\ref{xyz}) is isomorphic to $\SL(2,5)=2A_5$;
this was checked  using GAP (by A. Hulpke) and using Magma.
Thus,  $x^3=y^3$ is the unique involution in $\SL(2,5)$,  so that  
$  (w^u)^3=w^3=(ww^u)^2=(w^uw)^2 $  in $J$.

Since $T$ is 2-homogeneous on $\Omega $,  the preceding proposition
now implies that 
$N:=\< \Omega\>\cong  2A_{p+2}$.   
Clearly, $T$ and $w$ normalize $N$, whence $N$ is normal in $J$.
So $J= NT$ and hence $|J| \le |2A_{p+2} ||T|$, as required.
\end{proof}

Since either $T$ or $T \times \mathbb{Z}/2$ can be generated by
a single conjugacy classe, we can add one extra relation to obtain: 

\begin{corollary}  
\label{2altp+2}
For any prime $p \equiv 11 \mod 12 , $   both
  $A_{p+2}$ or   $2 A_{p+2}$  have presentations
with $2$ generators and $4$ relations.
\end{corollary}

For $A_{p+2}$, this is already proved in \cite{GKKL3}. 

 Proposition~ \ref{SxT} now
implies that there is even a profinite presentation of $2A_{p+2}$ with 2 generators and 
only 3   relations, proving
part of Theorem~\ref{altp+2}(b) when $p \equiv 11 \mod 12$.
For the more general case $p \equiv 3 \mod 4$ we will need
 more tools (see Theorem~\ref{2Sn}).  

We finish this section by restating and generalizing some of our earlier results,  
as well as  \cite[Corollary 3.8]{GKKL3},   in terms of amalgamated products. 

\begin{lemma} \label{amalgam1}  Let  $p \equiv 3 \mod 4$ be  prime.  
Let $T$ be the subgroup of index $2$ in $\AGL(1,p)$.
Then 
$A_{p+2} \times T = X/N, $   where 
\begin{itemize}
\item[\rm (i)]
$X$
is the $($free$)$ amalgamated product of $A$ and $T$ with $A \cap T = C$
cyclic of order $(p-1)/2$  and
$A \cong \Z/(3) \times  C,$ and
\item[\rm (ii)]
$N$ is the normal closure in $X$ of a single element  $x^2,$  $x\in X$.
\end{itemize}
\end{lemma}
 
 \begin{proof}   
 Let $X$ be the given amalgamated product.  Write
 $A = \langle a \rangle \times  \langle c \rangle$ where $a^3=1$ and $c $  generates $C$.
 Let $T = \langle u, v  \rangle$ where $u^p=1$, $v$ normalizes 
 $\langle u \rangle$ and $v$ has order $(p-1)/2$.  We identify
 $v$ with $c$,  so that $X$ has the  presentation   
 $$X =\langle a,  u, v \mid a^3=v^{(p-1)/2}=[a,v]=1,  
 u^p=1, u^v=u^e  \rangle$$
 for an integer  $e$ of order $(p-1)/2$  mod $p$.  
 We identify $T$ with a subgroup of $A_p < A_{p+2}$ acting on $\F_p$, fixing $\{p+1, p+2\}$, and 
 such that  $v$ fixes $0 \in \F_p$.

There is a surjection   $\phi\colon X \to A_{p+2} \times T$
 sending
    $u \mapsto  (u,u)$,
   $v\mapsto (v,v)$ and  $a \mapsto (a_0,1)$ with $a_0= (0, p +1  ,  p + 2) \in A_{p+2}$. 
     We can identify $T$ with the subgroup
  $\langle u, v \rangle$   of $X$.

 Let $N :=\langle (x^2)^X\rangle $  with $x:=aa^u$.
Since $\phi((aa^u)^2) = (  ( a_0a_0^{u_0} )^2  ,1 )=1$,    $X/N$ surjects onto
 $A_{p+2} \times T$.  
 
 Since $v$ centralizes $a$, 
 as in the proof of Proposition~\ref{doublexT} 
we again see that $|a^T| = p$,  
 and hence that 
 $T$ acts transitively on the 2-element subsets   of $a^T$ since  $p \equiv 3 \mod 4$.  Thus,
 $(a_1a_2)^2 \in N$ for every pair of distinct elements $a_1,a_2 \in a^T$.
Clearly $\langle a^T \rangle $ is normal in $X$, so that 
$\langle a^T \rangle =  \langle a^X \rangle $.  By (\ref{carmichael}),
$\langle a^T \rangle
 \cong A_{p+2}$.   Since $X/\langle a^X\rangle \cong T$, 
 we have  $|X/N| \le |A_{p+2}||T|$ and hence $X/N \cong A_{p+2} \times T$,
 as claimed.%
 \end{proof}

\section{More cohomology} \label{coho2}

We first prove a result for cohomology of amalgamated products (by which we will always mean {\em free}  amalgamated products).
One can prove a more precise version, but we will be  only need 
that $H^2(G,M)=0$ in restricted  situations. 

\begin{lemma}  \label{amalgam} Let $G$ be the  amalgamated product of the groups
$A$ and $B$ over $C$.  Let $M$ be a finite-dimensional $kG$-module.
Then 
$$
\dim H^2(G,M) \le \dim H^2(A,M) + \dim H^2(B,M) + \dim H^1(C,M).
$$
\end{lemma}

\begin{proof}  Let $U$ be the kernel of the natural map 
$H^2(G,M)\to H^2(A,M) \oplus  H^2(B,M)$. 
Clearly $\dim H^2(G,M) -\dim U \le \dim H^2(A,M) + \dim H^2(B,M)$. 
Thus,  it suffices
to show that there is an embedding of $U$ into $H^1(C,M)$.
\vspace{-2pt}

Let $u \in U$.  There is a corresponding  extension  
$1  \rightarrow M \rightarrow E  \stackrel{f}{\rightarrow} G \rightarrow 1$, and
$f$ splits over both $A$ and $B$: there are injections 
$\psi_A\colon A\to E$  and  $\psi_B\colon B\to E$  such that
\begin{equation}\label{embeddings}
\mbox { $f\psi_A=1_A~$  and $~f\psi_B=1_B$.}
\end{equation}
Let $A_1:=\psi_A(A)$  and  $B_1:=\psi_B(B)$.

The maps  $ \psi_A$   and $ \psi_B$ produce splittings of 
$1  \rightarrow M \rightarrow f^{-1}(C) \rightarrow C\rightarrow 1$,
and hence also define derivations  $\delta_A$ and $\delta_B$ from $C$ to $M$.
Replacing $A _1$  by $A_1^m$ with $m\in M$  changes $\delta_A$ by an inner derivation, and hence 
we obtain  a well-defined linear map $U \rightarrow H^1(C,M)$.  
Consequently, 
 $u\mapsto \delta :=\delta_A - \delta_B$   induces a linear map  $ U\to H^1(C,M)$.

We claim that this map   is injective.  Assume that $\delta$ is an  inner derivation
on $C$.  This means that  the splitting  $\psi_A|_C$ is obtained from $\psi_B|_C$ 
by  conjugating by an element of $M$.  
Therefore, replacing $B_1$ by  a conjugate we may assume that
$\psi_A|_C= \psi_B|_C$.
By the universal property of $G=A*_CB$, there is a homomorphism 
$\psi\colon G\to E$ such that 
$\psi |_A = \psi_A$  and $\psi |_B = \psi_B$.  Since $f\psi = 1_G$ by 
(\ref{embeddings}), this completes the proof.
 \end{proof}

We will use the previous result in the following form:

\begin{corollary} 
\label{H2=0}
 Let $X$ be an amalgamated product of finite groups
$A$ and $B$ of order prime to $p$.  If $V$ is a finite-dimensional
$kX$-module over a field  $k$  of characteristic $p,$   then $H^2(X,V)=0$.
\end{corollary}

\begin{lemma} \label{quotient}  Let $G = X/N$ and let $M$ be a $kX$-module
for some field $k$,  with $N$ acting trivially.  View $M$
as a $kG$-module.
Assume that $N$ can be generated by $s$ elements as 
a normal subgroup of $G$.  Then
$$\dim H^2(G,M) \le \dim H^2(X,M) + s \dim M.$$
\end{lemma}

\begin{proof}   View  By the inflation restriction 
sequence (Lemma \ref{infres}),
there is an exact sequence
$$
H^1(N,M)^X \rightarrow H^2(X/N, M) \rightarrow H^2(X,M).
$$
Since $N$ acts trivially on $M$,  $H^1(N,M)^X \cong \hom_X(N,M)
\cong \hom_G(N/[N,N], M)$.   Since $N$ can be generated as a normal
subgroup by $s$ elements,  $N/[N,N]$ can be generated as a $G$-module
by $s$ elements, whence $\dim H^1(N,M)^X \le s \dim M$.
\end{proof}

Lemmas \ref{amalgam} and \ref{quotient}  imply the following (using $X =A*B$ in Lemma~\ref{quotient}):

\begin{lemma}  Let $G$ be a group with subgroups $A$ and $B$ such
that $G = \langle A, B \mid w_i = 1, \ 1 \le i \le t \rangle$ for words  $w_i$ in
$A \cup B$.    Let  $M$ a finite-dimensional $kG$-module
over a field $k$.
Then 
$$\dim H^2(G,M) \le \dim H^2(A,M) + \dim H^2(B,M) + t \dim M.$$
\end{lemma}
 
We will also need the following result   \cite[Lemma~4.1(3)]{GKKL2}  about covering groups 
 (this is stated  there  for
quasisimple groups, but the proof does not use this):

\begin{lemma} \label{covering}   Let $G$ be a  finite group with
center $Z$. 
Let $M$ be a nontrivial irreducible $kG$-module with $Z$ trivial on $M, $ 
where $k$ is a field of characteristic $r$. 
Then 
$$\dim H^2(G,M) \le \dim H^2(G/Z,M) + c \dim H^1(G,M),$$ 
where $c$ is the $r$-rank of $Z$.
\end{lemma}

The next observation is trivial:

\begin{lemma} \label{covering2}  Let $G$ be a finite perfect
group and $\tilde{G}$ its  universal cover.  If $\tilde{G}$
is proficient, then so is $G$.
\end{lemma}

\begin{corollary} \label{cov-cyclic}
 Let $S$ be a finite perfect group with cyclic Schur multiplier,
 trivial center and universal cover $\tilde{S}$.
 If $\dim H^2(S,M) \le \dim M$ for all irreducible $S$-modules $M,$
 then any central quotient of $\tilde{S}$ is proficient.
 \end{corollary}
 
 \begin{proof}  
  By the previous lemma, it suffices to show that
$\tilde{S}$ is proficient.  Suppose that $M$ is an irreducible
$\F_p\tilde{S}$-module.  If $Z(\tilde{S})$ acts nontrivial
on $M$, then $H^j(\tilde{S}, M)=0$ for all $j \ge 0$ by
\cite[Corollary 3.12]{GKKL2}.  In particular, $\nu_2(M)=0$.  

Otherwise, we may also view $M$ as an $S$-module.
By Lemma \ref{covering} and the hypotheses, 
$\dim H^2(\tilde{S},M)-\dim H^1(\tilde{S},M)
\le \dim H^2(S,M) \le \dim M$ for any irreducible $\F_p\tilde{S}$-module $M$.
Thus $\tilde{S}$ is proficient   by    (\ref{prof ineq2})
 \end{proof}

\section{Cohomology of some alternating groups} 
\label{altsec}
In this section, we fix
a prime $p \equiv 3 \mod 4$  and consider   $A_{p+2}$ and $S_{p+2}$. 
We first improve a bound \cite[Theorem 6.2]{GKKL2} for $H^2$:

\begin{theorem}
Set $G=A_{p+2}$.  Let   $M$  be   a $kG$-module
for  a field  $k$  of characteristic $r$.
\begin{enumerate}
\item   $\dim H^2(G,M) \le \dim M,$ with
equality if and only if $r=2$ and $G$ acts trivially on $M$.
\item If $M$ is nontrivial and irreducible, then
$\dim H^2(G,M) \le \frac{p-1}{p+1}\dim M$.
\item Both $A_{p+2}$ and its double cover are proficient.
\end{enumerate}
\end{theorem}

\begin{proof}   If $r > 3$, then
$\dim H^2(G,M) \le (\dim M)/(r-2)
\le \frac{p-1}{p+1} \dim M$   by \cite[Theorem 6.4]{GKKL2}.  Suppose that $r=3$.  
Since $p+2 \ne 7$ or $8$,
$\dim H^2(G,M) \le (21/25) \dim M \le \frac{p-1}{p+1} \dim M$ unless
$p=3$ or $7$, by \cite[Theorem 6.5]{GKKL2}.  If $p=3$, by inspection  $\dim H^2(G,M) \le 1$
and the result holds.  If $p=7$, then
$\dim H^2(G,M) \le (3/5) \dim M$ by \cite[Theorem 6.5]{GKKL2},
 and the result  again  holds.

So assume that $r=2$.
Since the Schur multiplier has order 2,   we may also assume   that $M$ is
irreducible and nontrivial, so we need to  consider (2).

 Write $G = X/N$ with $A,T$   and $x^2$ as in   Lemma \ref{amalgam1}.
 Since $A$ and $T$ have odd order,  $ H^2(X,M)=0$  by Corollary~\ref{H2=0}.
By  Lemma \ref{infres}, we then have
  $\dim H^2(G,M) \le \dim \hom_G(N/[N,N],M)$.
  Since $N/[N,N]$ is a cyclic $G$-module, the last
term is at most $\dim M$.  However, since a (normal) generator  $x^2[N,N]$
 of $N/[N,N]$
is fixed by the nontrivial element $xN$,  we see that
$ \dim \hom_G(N/[N,N],M)$ is at most
 the dimension of the space $C_M(x)$  of   fixed points of $x$ on $M$.

If $p=3$, one can verify (2) using MAGMA.
 If $p > 3$,
 by \cite[Lemma 6.1]{gursaxl}   $G$ can be generated by
$(p+1)/2$ conjugates of any nontrivial element.
In particular, $G$ is generated by conjugates $x_1 , \dots,x_{(p+1)/2}$ of $x$.
Then $\dim M = {\rm codim\ } \cap_iC_M(x_i)\le \frac{p + 1}{2}  {\rm codim\ }  C_M(x)$, so that
$\dim C_M(x)\le  \dim M - \frac{2}{p + 1} (\dim M)=\frac{p-1}{p+1}\dim M$.

This proves (2) and hence (1), and (3) follows from Corollary \ref{cov-cyclic}.
\end{proof}

We have now proved Theorem~\ref{altp+2}(b).
We still need to prove Theorem~\ref{altp+2}(c),
for which we need more information concerning
 $A_{p+2}$.
We first record a special case of \cite[Theorem 1]{gurkim}.

\begin{lemma}  If $M$ is a  $kA_{p+2}$-module for any field $k,  $  then  
$\dim H^1(A_{p+2},M) \le (\dim M)/(p-1)$. 
\end{lemma}

The same bound holds for $kS_{p+2}$-modules
such that   $A_{p+2}$   has
 no fixed points -- for then $H^1(S_{p+2},M)$
embeds into $H^1(A_{p+2},M)$ (see \cite[Lemma 3.8(1)]{GKKL2}).

Combining the previous two  results gives:

\begin{lemma} \label{h2+h1}  Let $k$ be a field of characteristic
$r$. 
Let $M$ be a nontrivial irreducible $kA_{p+2}$-module.
Then
$$
\dim H^2(A_{p+2},M) + \dim H^1(A_{p+2},M) \le \frac{p^2 -p+2}{p^2-1}\dim M \le \dim M. 
$$
\end{lemma}

Note that unless $p=3$, we have a strict inequality above.
 We can now prove:

\begin{theorem}  
\label{H2 estimate Alt Sym}
Let $k$ be a field of characteristic $r$. 
Let $G$ be either $S_{p+2}$ or  $2A_{p+2}$.
If $M$ is a  nontrivial  irreducible $kG$-module$,$  then
$\dim H^2(G,M) \le \dim M.$
In particular$,$   $S_{p+2}$ and  
$2A_{p+2}$ are proficient.  
\end{theorem}

\begin{proof}
First suppose that $r \ne 2$.
If $G = 2A_{p+2}$, it follows by \cite[Theorem 6.2]{GKKL2}
that $\dim H^2(G,M) < \dim M$ (noting that the trivial module
has trivial $H^2$).  
If   $G=S_{p+2}$, then $M$ restricted to $A_{p+2}$ is either 
irreducible or the direct sum of $2$  irreducible modules.  Since
$r$ is odd, $H^2(S_{p+2},M)$ embeds in $H^2(A_{p+2}, M)$ by
\cite[p.~91]{relmod},  and  
the result follows.

Now consider $r=2$.   First suppose that $G=S_{p+2}$ with
$M$ a nontrivial irreducible $S_{p+2}$-module.   Then $M$ is a direct
sum of nontrivial irreducible $A_{p+2}$-modules.
 By \cite[Lemma 2.8(2)]{GKKL2},$~\dim H^2(G,M) \le \dim H^2(A_{p+2},M)
+ \dim H^1(A_{p+2}, M)$. By Lemma \ref{h2+h1}, this implies that 
$\dim H^2(G,M) \le �  \dim M $.  

 Similarly, using  Lemma \ref{covering} and Lemma \ref{h2+h1},  
if  $G=2A_{p+2}$
then 
$$\dim H^2(G,M) \le \dim H^2(A_{p+2},M)
+ \dim H^1(A_{p+2}, M)  \le \dim M.$$
The result follows by  Corollary~ \ref{cov-cyclic}.
\end{proof}

  Moreover, any double cover of $S_{p+2}$  
  (which is nonsplit when restricted to $A_{p+2}$)   is proficient: 

\begin{theorem} \label{2Sn}
Let $X$ be a double cover of $S_{p+2}$ that
is nonsplit over $A_{p+2}$.  Then $X$  is proficient$:$   it has a profinite 
presentation with $2$ generators and $2$ relations. 
\end{theorem}

\begin{proof} 
Let $Z=Z(X)$, and so $|Z|=2$.  Let $k$ be a field of characteristic
$r$ with $M$ a nontrivial irreducible $kX$-module.  If $r \ne 2$,
then the restriction map from $H^2(X,M)$ to $H^2(2A_{p+2},M)$
is injective by  \cite[p.~91]{relmod}.    Arguing as above,
we see that $\dim H^2(2A_{p+2},M)=0$ if $Z$ acts nontrivially
and that $\dim H^2((2A_{p+2},M) = \dim H^2(A_{p+2},M) \le \dim M$
if $Z$ acts trivially.   Thus, $\nu_2(M) \le 1$ for all such $M$.

If $r=2$,  then $Z$ is trivial on $M$. 
Then 
$\dim H^2(X,M) \le \dim H^2(S_{p+2},M) + \dim H^1(X,M)$
by Lemma \ref{covering}.
Thus, by Corollary~\ref{cov-cyclic} and definition, $\nu_2(M) \le 1$ for any nontrivial 
irreducible $kX$-module.

Suppose that $M =k$ is the trivial module. 
   We claim that $\dim H^2(X,k)=1$. 
     Since the derived subgroup  $Y$ 
   of $X$ is the universal
cover of $A_{p+2}$, we have $H^2(Y,k)=0;$  and 
since $Y$ is perfect, 
$H^1(Y,k)=0$.   By \cite[Lemma 3.11]{GKKL2},  
these imply the claim.  Clearly, $\dim H^1(X,k)=1$,
whence the result follows.  
   Thus, considering  all cases  we have $\max_{p,M}\nu_2(M)\le 1$,   and hence $\hat r_2(X)=2$
 by 
(\ref{prof ineq2}).%
\end{proof}

{\noindent  \em Completion of the proof of}  Theorem~\ref{altp+2}.    We have already proved
(a).  Parts (b) and (c) follow from the two previous results.   We now
prove (d), (e)  and (f).   Let $k$ be a field. 
 
(d)  By \cite[Theorems 7.2, 7.3]{GKKL2},
if  $q \ge 4$  then  $\dim H^2(G,M) \le \dim M$ for any
$k\SL(2,q)$-module  $M$ and so $\SL(2,q)$ is proficient
by Corollary~\ref{cov-cyclic}.

(e),(f)  Let $G=\PSL(3,q)$ with $q \equiv 1 \mod 3$ 
or $\PSL(4,q)$ with $q$ odd.  Then 
$G$ has a nontrivial Schur multiplier.  Thus,  by (\ref{rhat}) it suffices
to observe that $\dim H^2(G,M) \le 2 \dim M$ for
any irreducible $kG$-module $M$.  This is \cite[Theorem E]{GKKL2}.
\qed

\medskip
Finally, we show how our methods can be used to give
very good estimates on  some second cohomology groups:
we give a new and simpler proof of a result of Kleshchev and Premet
\cite{KP}.     
  

\begin{theorem} \label{heart}  Let $G=A_n, n > 4$.   Let $M$ be the nontrivial
irreducible composition factor of the permutation module   $P$
 of dimension $n$
over a field $k$ of characteristic $r$.
Assume that $n > 5$ if $r=5$ and $n > 9$ if $r=3$.
Then $H^2(G,M)=0$.
\end{theorem}

\begin{proof}  We will need a variant of the  presentation    (\ref{carmichael}) for $G$.    Let $I = \{1, \ldots, n\}$.
If $J$ is a subset of $I$, let $G_J$ be  the  
subgroup which acts on $I \setminus{J}$ as the alternating group and is trivial
on $J$.

Let $X$ be the free amalgamated product of $G_1$ and $G_n$
over $G_{1,n}$.
Let $R$ be the normal subgroup of $X$ generated by the element $w:=(uv)^2$,
where $u = (1 ,3 , 4) \in G_{n}$ and $v = (3 , 4 , n) \in G_1$.  Let $\Omega$ be the set
of $3$-cycles of the form $(i ,3 , 4)$.    Note that $u,v \in \Omega$ and that every other
element of $\Omega$ is in $G_{1,n}$.  Then
$X/R \cong A_n$   by  (\ref{carmichael}).

If $n > 5$, 
let $Y$ be the   free amalgamated product of $G_{1,2}$ and $G_{2,n}$
over $G_{1,2,n}$.  We may view $Y$ as a subgroup of $X$.  Then
$w \in Y$,  and  the image of $Y$ in $X/R \cong A_n$ is $A_{n-1}$.
Let $S$ be the normal closure of $w$ in $Y$.  Again, by   (\ref{carmichael}),
$Y/S \cong A_{n-1}$.  Since $S \le R$ we have  $R \cap Y = S$.

First suppose that $r$ does not divide $n$.  
Then $P = k \oplus M$. 
By Shapiro's Lemma  (e.g., \cite[Lemma~3.3]{GKKL2}), 
$\dim H^2(G,P)=\dim H^2(A_{n-1},k)$. Thus,
$\dim H^2(G,M)= \dim H^2(A_{n-1},k) - \dim H^2(G,k)$.
If $r=2$, both of the latter  quantities are $1$.  
If $r > 3$, or if  $r=3$ and $n > 9$, then
both of those quantities are $0$ 
(since  $M(A_m)=\Z_2$  if $m=5$ or $m > 7$
and  $M(A_6)=M(A_7)=\Z_6$).
Thus, $H^2(G,M)=0$.

Now assume that $r|n$.  By our hypotheses, this implies that $n > 5$
(and $n \ge 12$ if $r=3$).  
We view $M$ as a $kX$-module with $R$ acting trivially.  Note
that $M$ restricted to $A_{n-1}$ is the nontrivial composition
factor of the permutation module for $A_{n-1}$.  Thus, by induction,
$H^2(A_{n-1},M)=0$.  Also, by Frobenius reciprocity,
$H^1(A_{n-1},M)=H^1(A_{n-2},k)=0$.
By the inflation restriction sequence,
$$
0 \rightarrow H^1(Y/S,M) \rightarrow H^1(Y,M) \rightarrow \hom_Y(S,M)
\rightarrow H^2(Y/S,M) \rightarrow H^2(Y,M).
$$ 
Since $Y/S= A_{n-1}$, we know  that
$H^i(Y/S,M)=0$ for $i=1,2$.   Thus, $H^1(Y,M) \cong \hom_Y(S,M)$.

We claim that $H^1(Y,M)=0$, and so also $\hom_Y(S,M)=0$. 
Let $D: = \mathrm{Der}(Y,M)$.  Let $f:D \rightarrow \mathrm{Der}(G_{1,2},M)$
be the restriction map.  Note that $M$ is the permutation module for $G_{1,2}
\cong A_{n-2}$.   Thus, $H^1(G_{1,2},M)=0$ and so any element of
$\mathrm{Der}(G_{1,2},M)$ is inner.  Since $G_{1,2}$ has a $1$-dimensional
fixed space on $M$, it follows that the image of $f$ has dimension $n-3$
(clearly the map is onto and the space of inner derivations for $H$ acting
on $M$ is isomorphic to $M/M^H$). 

Let $K = \mathrm{ker}(f)$.  Since $Y = \langle G_{1,2}, G_{2,n} \rangle$, the
restriction mapping $f_1: K \rightarrow \mathrm{Der}(G_{2,n},M)$
is injective.  As already noted, $\mathrm{Der}(G_{2,n},M)$ consists of inner
derivations.  Thus, the image of $f_1$ are those inner derivations of $G_{2,n}$
which vanish on $G_{1,2,n}$.   Since $M$ is the permutation module
for $G_{1,2}$, it follows that $G_{1,2}$ has a $1$-dimensional fixed
space and $G_{1,2,n}$ has a $2$-dimensional fixed space.  Thus,
the image of $f_1$ is $1$-dimensional.  Hence $\dim D = n-2$.
Since $Y$ acts irreducibly and nontrivially on $M$,  the space
of inner derivations of $Y$ on $M$ also has dimension $n-2$.
Thus, $\mathrm{Der}(Y,M)$ consists of inner derivations and so $H^1(Y,M)=0$,
as claimed.

Also by the inflation restriction sequence,  
$$
0 \rightarrow H^1(X/R,M) \rightarrow H^1(X,M) \rightarrow \hom_X(R,M)
\rightarrow H^2(X/R,M) \rightarrow H^2(X,M).
$$
Since $H^2(A_{n-1},M)=H^1(A_{n-2},M)=0$, it follows by Lemma \ref{amalgam} that
$H^2(X,M)=0$.   Thus, to complete the proof, it suffices to show
that $\hom_X(R,M)=0$. This follows since
  the restriction mapping   $\hom_X(R,M)\to \hom_Y(S,M)$ is injective 
(as $w \in S$ generates $R$ as a normal subgroup of $X$).
\end{proof}

It is straightforward to compute $H^2(G,M)$ in the cases omitted in 
the theorem. In fact, they are all $1$-dimensional except that
$H^2(A_7,M)=0$ in characteristic $3$.

\section{Non-proficient groups} \label{non}

There have been many constructions of non-proficient groups,
starting with Swan \cite{swan}.  See also \cite{GrK} and  \cite{kovacs}.
 
Let $V$ be a  nontrivial irreducible   $\F_pH$-module for  a finite perfect
group $H$. 
Let $W=V^e$ for some positive integer $e$,  so  that $W\semi H$ is perfect.
 Let $Y$ be the universal
cover of $WH$.  Then:

\begin{proposition}  If $e > \dim V, $  then $G:=Y \times Y$ is not proficient.
\end{proposition}

\begin{proof}  Consider the  irreducible $\F_pG$-module   $M = V \otimes V$.
By the K\"unneth formula,  $\dim H^2(G,M) = e^2 > \dim M  $
and $H^1(G,M)=0$.   
Since $G$ is perfect and $Y$ has trivial Schur multiplier, so 
does $G$.  Thus, $\nu_2(M) > 1 = \nu_2(\F_r) $ for any prime $r$
and so $G$ is not proficient  by (\ref{prof ineq2}).
\end{proof}


Similarly:

\begin{proposition} Any finite group  $S$ is a direct
summand of a finite non-proficient group.
\end{proposition}

\begin{proof}   Let $G, M$ and $e$  be as in the previous proposition.  We may also
assume that $e$ is sufficiently large so that $\nu_2(M) >  d(M(S))$.  

Let $X:=S \times G$.  We may view $M$ as an $\F_pX$-module.  Since $p$ does not
divide the order of $S$, by the K\"unneth formula
$\dim H^i(X,M)=\dim H^i(G,M)$.  Thus, the computation of $\nu_2(M)$ is the same for
$X$  and $G$.   In particular, $\nu_2(M) > d(M(S))=d(M(X))$.  Thus, 
$X$ is not proficient.  
\end{proof}

{\noindent\bf  More examples, including non-proficient  solvable groups:}
We now  give additional examples of non-proficient
groups.   

We first compute $H^2$ for certain semidirect products.

\begin{lemma}  \label{calc}  Let $p$ be a prime.
Let $G$ be a finite group with a normal  elementary 
abelian $p$-subgroup $L$.  Assume that $G/L$ has order prime to $p$.
Let $r$ be a prime and $U$ be an irreducible $\F_rG$-module. 
\begin{enumerate}
\item  If $r=p$, then  $\dim H^2(G,U) = \dim \hom_G(L,U) + \dim \hom_G(\wedge^2(L),U)$.
\item  If $r \ne p$ and $U^L = 0$, then $H^j(G,U)=0$ for all $j \ge 0$.
\item  If $r \ne p$ and $U^L  \ne 0$, then $H^j(G,U) \cong H^j(G/L,U)$ for all $j \ge 0$.
\end{enumerate}
\end{lemma}

\begin{proof}   Note that $G=L\semi H$ for some subgroup $H$, by the Schur-Zassenhaus Theorem. 
 
 First assume  that $r=p$.
  Let $w_1, \ldots, w_d$
be a basis for $L$.   Let $X$ be the universal nilpotent group of class $2$ generated
by  elements  $x_1, \ldots, x_d$  satisfying $x_i^{p^2}=1$.   Since
$H$ has order prime to $p$, $H$ acts naturally on $X$
so as to make 
$x_i \to w_i $  induce an $H$-equivariant map.
Note that
$\wedge^2(L) \cong [X,X] \le Y:=Z(X)$ (as $H$-modules)�.  
If $X_1: = \langle x_1^p, \ldots, x_d^p \rangle
\le Y$,  then $X_1 \cong L$ as $H$-modules.   
Clearly, $X/X_1[X,X] \cong L$ and so $Y=[X,X] \times X_1$.  In particular,
$Y \cong \wedge^2(L) \oplus L$ as $H$-modules. 
 
Consider any element of $H^2(L,U)^G$.  By the universality of $X$,
this corresponds to an extension 
$1 \rightarrow Y/M \rightarrow X/M \rightarrow L \rightarrow 1$
with $M$ an $H$-invariant subgroup of $Y$ with $Y/M \cong U$
as $H$-modules.  Clearly this lifts to an element of
$H^2(G,U)$, giving  a map   $H^2(L,U)^G \rightarrow H^2(G,U)$.
Composing with the restriction map gives the identity on $H^2(L,U)^G$.
Since $p$ does not divide $|G/L|$,  restriction is an injection  and so
$H^2(L,U)^G \cong H^2(G,U)$.

Since $G/L$ has order prime to $p$,  taking fixed points in
(i) of Lemma~\ref{3.16-2} gives $ H^2(L,U)^G = \hom_H(L,U) + \hom_H(\wedge^2(L),U)$,
and (1) follows. 

Finally, if  $r \ne p$, then (2) and (3) are  \cite[Corollary 3.12]{GKKL2}.
\end{proof}


 Let $H$ be a finite
group, and let $V$ be an irreducible $\F_pH$-module for some prime $p$
not dividing the order of $H$.  Let $W=V^e$ and set $G=W\semi H$.  
Assume that  $V$ is not self-dual and $\dim V = s > 1$.     Let
$U$ be an irreducible $\F_pH$-module that is a homomorphic image
of $\wedge^2(V)$.   Since $V$ is not self dual,  $U$ is nontrivial.

By Lemma \ref{calc},  
$\dim H^2(G,U) = \dim \hom_H(V^e,U) + \dim \hom_H(\wedge^2(V^e),U)$.
Also, by Lemma \ref{h2-h1 nontrivial}, $\dim H^1(G,U) = \dim \hom_H(V^e,U)$.

Thus, $\dim H^2(G,U) - \dim H^1(G,U) = \dim \hom_H(\wedge^2(V^e),U)$.
Since $U$ is a homomorphic image of $\wedge^2(V)$, its multiplicity
as a composition factor in $\wedge^2(V^e) = \wedge^2(V)^e + (V \otimes V)^{e(e-1)/2}$
is at least $e(e+1)/2$.   Thus, 
$$
\nu_2(U) \ge \frac{e(e+1)}{2d}  > \Big(\frac{e}{s} \Big)^2,
$$
where $d = \dim U$.  
 
Since $\F_p$ is not an image of either $V$ or $\wedge^2(V)$, it follows
by Lemma \ref{calc} that $H^2(G,\F_p)=0$.  Thus $p$ does not divide
$M(H)$ or $M(G)$. 
Also, by Lemma \ref{calc}, if $r \ne p$ then  $H^i(G, \F_r)=H^i(H,\F_r)$
 for all $i$. 
Hence, by Lemma \ref{schur}, $d(M(G))=d(M(H))$.
Since $G$ is not proficient as long as $\nu_2(U) > d(M(H)) + 1$,
we see that $G$ is not proficient for $e$ sufficiently large. 
In particular, if $H$ has a trivial Schur multiplier,
then $G$ is not proficient as long as $e(e+1) > 2d$.  

Note that $G$ is solvable if and only if $H$ is solvable.

We can  be a bit more precise.  The argument above shows that
$$ 
\hat{r}(G) - d(G) = \max\{\hat{r}(H)- d(H), \nu_2(U)-1  \}.
$$
where $U$ ranges over all $\F_pH$ composition factors of $ V \otimes V$.

One can also compute $d(G)$ easily.  If $s' = \dim_E V$, where 
$E$ is the field $\mathrm{End}_{\F_pH}(V)$, then  
  \cite[Corollary 2]{AG} implies
that  
$$
d(G) =  \max \{ d(H), 2 + \lfloor \frac{e-1}{s'} \rfloor \}.
$$

\smallskip
\smallskip

{\noindent   
\bf Corrections:} 
Finally,  we take this opportunity to correct
two minor errors in \cite{GKKL2} pointed out to us by Serre.

The first is    \cite[Lemma 3.11]{GKKL2} (and as restated in
\cite[Lemma 3.12(i)]{GKKL2}), which
we quoted incorrectly from  \cite{Ba}.
 The correct hypothesis is that
$H^i(N,M)=0$ for $0 < i < r$, which always held whenever
the result was applied.

The second is   \cite[Lemma 3.16]{GKKL2}, the correct version of which is:

\begin{lemma} \label{3.16-2}  Let $G$ be a finite group with a normal  
abelian $p$-subgroup $L$.  
Let $L[p]$ denote the $p$-torsion subgroup of $L$.
Let $V$ be an irreducible $\F_pG$-module. 
\begin{enumerate}
\item There is an exact sequence of $G$-modules$,$
$$
0 \rightarrow \mathrm{Ext}_{\mathbb{Z}}(L,V) \rightarrow H^2(L,V) \rightarrow 
\wedge^2( L^*) \otimes V \rightarrow 0.
$$
\item 
$\dim H^2(L,V)^G \le \dim (L[p]^* \otimes V)^G 
+ \dim_F (\wedge^2 (L/pL)^* \otimes V)^G$.
\item If $G=L,$  then $\dim H^2(G,\F_p) =  d(d+1)/2$ where $d=d(G)$.
\end{enumerate}
\end{lemma}

The only change is (2), where $L[p]$ replaces $L/pL$.  Again, this has
no effect on the proofs in \cite{GKKL2}.

\smallskip
\smallskip

{\noindent   
\bf Acknowledgments:} We thank J.-P. Serre for pointing out the errors just discussed, 
and  A. Hulpke for his assistance with
Proposition~\ref{doublexT}. We also thank the referee for his careful reading
and useful comments.

\end{document}